\documentclass[12pt]{amsart}
\usepackage{geometry}   %????????
\usepackage[colorlinks,citecolor = red, linkcolor=blue,hyperindex]{hyperref}
\usepackage{euscript,eufrak,verbatim, mathrsfs}
\usepackage[psamsfonts]{amssymb}
\usepackage{bbm}
\usepackage{graphicx}
 \usepackage{float}

 \usepackage[all, cmtip]{xy}

\usepackage{upref, xcolor, dsfont}
\usepackage{amsfonts,amsmath,amstext,amsbsy, amsopn,amsthm}
\usepackage{enumerate}

\usepackage{url}

\usepackage{mathtools}
\usepackage{bookmark}

 \usepackage{euscript}
\usepackage{helvet}         % selects\textbf{\textbf{•}} Helvetica as sans-serif font
\usepackage{courier}        % selects Courier as typewriter font
\usepackage{type1cm}        % activate if the above 3 fonts are
\usepackage{multicol}        % used for the two-column index
\usepackage[bottom]{footmisc}% places footnotes at page bottom

\newtheorem{theorem}{Theorem}[section]
\newtheorem*{theorem*}{Theorem B} 
\newtheorem{lemma}[theorem]{Lemma}

\newtheorem*{definition*}{Definition}
\newtheorem*{remark*}{Remark}

\newtheorem*{observation*}{Observation}

\newtheorem*{assumption*}{Assumption}
\newtheorem*{question*}{Question}

\newcommand{\N}{\mathbb{N}}

\newcommand{\C}{\mathbb{C}}
\newcommand{\E}{\mathbb{E}}

\newcommand{\Cov}{\mathrm{Cov}}

\begin{document}

\title[Extension of positive definite kernels]{On a result of Bo{\.z}ejko on extension of positive definite kernels}

\author%[authorlabel1]
{Yanqi Qiu}
\address%[authorlabel1]
{Yanqi QIU: Institute of Mathematics and Hua Loo-Keng Key Laboratory of Mathematics, AMSS, Chinese Academy of Sciences, Beijing 100190, China.
}
\email{yanqi.qiu@hotmail.com; yanqi.qiu@amss.ac.cn}

\thanks{This research is supported by grants NSFC Y7116335K1,  NSFC 11801547 and NSFC 11688101 of National Natural Science Foundation of China.}

\begin{abstract}
A conceptual proof of the result of Bo{\.z}ejko on extension of positive definite kernels is given. 
\end{abstract}


\maketitle

\setcounter{equation}{0}

\section{Introduction}

Let  $\C$ be the field of complex numbers.  In this note, by {\it positive definite}, we mean {\it non-negative definite}. Recall that a map $K: \Sigma \times \Sigma \rightarrow \C$ is called a positive definite kernel, if for each $k\in \N$, each choice of elements $\sigma_1, \cdots, \sigma_k \in \Sigma$, the square-matrix $[ K(\sigma_i, \sigma_j)]_{1\le i, j \le k}$ is  positive definite.

A beautiful result of Bo\.{z}ejko \cite{Bozejko-89} on the extension of positive definite kernels on union of two sets is as follows.  Let $K_i: \Sigma_i \times \Sigma_i \rightarrow \C$ be a kernel on a set $\Sigma_i$ ($i = 1, 2$). Assume that the intersection $\Sigma_1 \cap \Sigma_2 = \{x_0\}$ is a singleton and $K_1(x_0, x_0) = K_2(x_0, x_0)  = 1$. The {\it Markov product} of $K_1$ and $K_2$, denoted by $K_1*_{x_0} K_2$,  is a kernel on the union $\Sigma_1 \cup \Sigma_2$, defined by 
\begin{itemize}
\item $[K_1*_{x_0} K_2]\Big|_{\Sigma_i\times \Sigma_i} = K_i ( i = 1, 2);$
\item (Markov property) For $\sigma_i \in \Sigma_i(i = 1, 2)$, 
\[
[K_1*_{x_0} K_2](\sigma_1, \sigma_2)= K_1(\sigma_1, x_0) K_2(x_0, \sigma_2)
\]
and 
\[
 [K_1*_{x_0} K_2](\sigma_2, \sigma_1)= \overline{[K_1*_{x_0} K_2](\sigma_1, \sigma_2)}.
\]
\end{itemize}

\begin{theorem}[{Bo\.{z}ejko \cite[Theorem 4.1]{Bozejko-89}}]\label{thm_Bozej}
Let $\Sigma_1, \Sigma_2$ be two sets such that the intersection $\Sigma_1 \cap \Sigma_2 = \{x_0\}$ is a singleton. 
Let  $K_1, K_2$ be two positive definite kernels on $\Sigma_1, \Sigma_2$  respectively. Assume that 
\[
K_i (\sigma_i, \sigma_i) = 1 \quad \text{for all $\sigma_i \in \Sigma_i \, (i = 1, 2)$.}
\]
Then the Markov product $K_1*_{x_0}K_2$ of $K_1$ and $K_2$ is also a positive definite kernel. 
\end{theorem}

 A slight improvement of Theorem \ref{thm_Bozej} is the following 

\begin{theorem}\label{thm_improve}
Let $\Sigma_1, \Sigma_2$ be two sets such that the intersection $\Sigma_1 \cap \Sigma_2 = \{x_0\}$ is a singleton. 
Let  $K_1, K_2$ be two positive definite kernels on $\Sigma_1, \Sigma_2$  respectively. Assume that 
\[
K_1 (x_0, x_0) = K_2(x_0, x_0) =  1. 
\]
Then the Markov product $K_1*_{x_0}K_2$ of $K_1$ and $K_2$ is also a positive definite kernel. 
\end{theorem}

The main purpose of this note is to give a conceptual  proof of Theorem \ref{thm_improve}. More precisely,   using basic results on Gaussian processes,  under the assumptions of Theorem \ref{thm_improve}, we will construct a family of random variables $(Y_\sigma)_{\sigma\in \Sigma_1 \cup \Sigma_2}$, not necessarily Gaussian, such that 
\[
[K_1*_{x_0} K_2 ] (\sigma, \tau) =  \E\big(Y_{\sigma} \overline{Y_{\tau}}\big), \quad \forall \sigma, \tau \in \Sigma_1 \cup \Sigma_2.
\]
The above representation of the kernel $K_1*_{x_0} K_2$ clearly implies that it is positive definite. 

\section{Conceptual proof of Theorem \ref{thm_improve}}
The following elementary lemmas will be  our main ingredients. 

\begin{lemma}\label{lem_pos}
Let $A$ be a positive definite $n\times n$ matrix. Then for a row vector $\alpha \in \C^n$, the matrix  
\[
T(\alpha) = \left[
\begin{array}{cc}
1 & \alpha
\\
\alpha^* & A
\end{array}
\right]
\]
is positive definite if and only if $A - \alpha^* \alpha$ is positive definite. 
\end{lemma}

\begin{proof}
Let $I_n$ denote the $n\times n$ identity matrix.  Clearly, $T(\alpha)$ is positive definite if and only if 
\[
 \left[
\begin{array}{cc}
1 & 0
\\
- \alpha^* & I_n
\end{array}
\right] \cdot T(\alpha) \cdot
\left[
\begin{array}{cc}
1 & -\alpha
\\
0 & I_n
\end{array}
\right] =  \left[
\begin{array}{cc}
1 & 0
\\
0 & A-\alpha^* \alpha
\end{array}
\right] 
\]
is positive definite. Therefore, $T(\alpha)$ is positive definite if and only if $A - \alpha^* \alpha$ is positive definite. 
\end{proof}

\begin{lemma}\label{lem_gauss}
Let  $K$ be a positive definite kernel on a finite set $S$. Assume that there exists $s_0\in S$ such that $K(s_0, s_0) = 1$. Then there exists a Gaussian process $(X_s)_{s \in S\setminus \{s_0\}}$ such that, by setting $X_{s_0} \equiv 1$, we have 
\begin{align}\label{K-X}
K(s, t) = \E (X_s \overline{X}_t), \quad  \forall s, t\in S. 
\end{align}
\end{lemma}

\begin{proof}
Let $\alpha \in \C^{S\setminus \{s_0\}}$ be the row vector defined by $\alpha(t) = K(s_0, t)$ for all $t\in S\setminus \{s_0\}$ and let $A \in \C^{(S\setminus \{s_0\})\times (S\setminus \{s_0\})}$ be the square matrix defined by $A(s, t) = K(s, t)$ for all $s, t \in S\setminus \{s_0\}$. By Lemma \ref{lem_pos} and the assumption that $K$ is positive definite (and thus $K(s_0, s)^*  = K(s, s_0)$), the following matrix 
\[
A-\alpha^{*} \alpha = \Big[ 
\begin{array}{cc}
K(s, t) -  K(s, s_0) K(s_0, t)
\end{array}
\Big]_{s, t \in S\setminus \{s_0\}}
\]
is positive definite. As a consequence, there exists a Gaussian process $(X_s)_{s\in S\setminus \{s_0\}}$ such that 
\begin{align}\label{cov-str}
\left\{
\begin{array}{ll}
\Cov(X_s, X_t)  = K(s, t) - K(s, s_0) K(s_0, t), &  \forall s, t \in S\setminus \{s_0\}
\vspace{2mm}
\\
\E(X_s)= K(s, s_0), &  \forall s \in S\setminus \{s_0\}
\end{array}
\right..
\end{align}
Set $X_{s_0} \equiv 1$, then \eqref{K-X} is equivalent to \eqref{cov-str}. This completes the proof of the lemma. 
\end{proof}

\begin{proof}[Conclusion of the proof of Theorem \ref{thm_improve}]
 Without loss of generality, we may assume that both $\Sigma_1$ and  $\Sigma_2$ are finite sets. By Lemma \ref{lem_gauss}, we can construct two Gaussian processes $(X_{\sigma_i}^{(i)})_{\sigma_i \in \Sigma_i \setminus  \{x_0\}}$ for $i  =1, 2$ such that, by setting $X_{x_0}^{(1)} = X_{x_0}^{(2)} \equiv 1$, for any $i =1, 2$,  we have 
\[
K_i (\sigma_i, \sigma_i')  = \E\Big(X_{\sigma_i}^{(i)} \overline{X_{\sigma_i'}^{(i)}}\Big),  \quad \forall \sigma_i, \sigma_i' \in \Sigma_i\setminus \{x_0\}. 
\]
Clearly, we may assume that the two Gaussian processes $(X_{\sigma_1}^{(1)})_{\sigma_1 \in \Sigma_1 \setminus  \{x_0\}}$ and $(X_{\sigma_2}^{(2)})_{\sigma_2 \in \Sigma_2 \setminus  \{x_0\}}$ are independent. Therefore,  since $X_{x_0}^{(1)} = X_{x_0}^{(2)} \equiv 1$, the two families of random variables $(X_{\sigma_1}^{(1)})_{\sigma_1 \in \Sigma_1}$ and $(X_{\sigma_2}^{(2)})_{\sigma_2 \in \Sigma_2}$ are independent. Let $(Y_\sigma)_{\sigma \in \Sigma_1 \cup \Sigma_2}$ be the family of random variables defined by 
\[
Y_\sigma= \left\{ 
\begin{array}{cc}
X_\sigma^{(1)}  & \text{if $\sigma \in \Sigma_1$}
\\ 
X_\sigma^{(2)}  & \text{if $\sigma \in \Sigma_2$}
\end{array}
\right..
\]
Note that this family $(Y_\sigma)_{\sigma \in \Sigma_1 \cup \Sigma_2}$ is indeed well-defined since we set $X_\sigma^{(1)}   = X_\sigma^{(2)} \equiv 1$ for $\sigma \in \Sigma_1 \cap \Sigma_2 = \{x_0\}$. 

Now  by the definition of $(Y_\sigma)_{\sigma \in \Sigma_1 \cup \Sigma_2}$,  it is immediate to check directly that 
\begin{itemize}
\item If $\sigma, \sigma'\in \Sigma_i$, then $\E\big(Y_\sigma \overline{Y_{\sigma'}}\big)   = K_i(\sigma, \sigma')$. 
\item If $\sigma_1\in  \Sigma_1,  \sigma_2 \in \Sigma_2$, then by the independence between $X_{\sigma_1}^{(1)}$ and $X_{\sigma_2}^{(2)}$ and by  the  second equality in \eqref{cov-str} (applied to $\Sigma_1$ and $\Sigma_2$ respectively), we have 
\begin{align*}
\E\big(Y_{\sigma_1} \overline{Y_{\sigma_2}}\big) & = \E\big(X_{\sigma_1}^{(1)} \overline{X_{\sigma_2}^{(2)}}\big)   = \E\big(X_{\sigma_1}^{(1)} \big) \E \big( \overline{X_{\sigma_2}^{(2)}}\big) 
\\
=&  K_1 (\sigma_1, x_0) \overline{K_2(\sigma_2, x_0)} =  K_1 (\sigma_1, x_0) K_2(x_0, \sigma_2).
\end{align*}
Moreover,  $\E\big(Y_{\sigma_2} \overline{Y_{\sigma_1}} \big)= \overline{\E\big(Y_{\sigma_1} \overline{Y_{\sigma_2}}\big)}$.
\end{itemize}
That is, recalling  the definition of the Markov product $K_1*_{x_0} K_2$ of the kernel $K_1$ and $K_2$, we obtain  
\begin{align}\label{kernel-in-p}
[K_1*_{x_0} K_2 ] (\sigma, \tau) =  \E\big(Y_{\sigma} \overline{Y_{\tau}}\big), \quad \forall \sigma, \tau \in \Sigma_1 \cup \Sigma_2. 
\end{align}
The equality \eqref{kernel-in-p} implies clearly that the Markov product $K_1*_{x_0} K_2$ is positive definite and this completes the whole proof of Theorem \ref{thm_improve}. 
\end{proof}

%\bibliography{mybib}
%\bibliographystyle{plain}

\end{document}